\documentclass[11pt,a4paper, final, twoside]{article}

\usepackage[backend=biber]{biblatex}
\usepackage{amsmath}
\usepackage{fancyhdr}
\usepackage{amsthm}
\usepackage{amsfonts}
\usepackage{amssymb}
\usepackage{amscd}
\usepackage{latexsym}
\usepackage{amsthm}
\usepackage{graphicx}
\usepackage{graphics}
\usepackage{tikz}
\usepackage{afterpage}
\usepackage[colorlinks=true, urlcolor=blue, linkcolor=black, citecolor=black]{hyperref}
\usepackage{color}
\newtheorem{theorem}{Theorem}[section]
\newtheorem{lemma}{Lemma}[section]

\newtheorem{example}{Example}[section]

\begin{document}
\hyphenpenalty=100000
\thispagestyle{plain} % Ensure the first page uses the plain style
\begin{flushleft}

{\Large \textbf{Classification of Associative Algebras Satisfying Quadratic Polynomial Identities}}\\[4mm]
{\large \textbf{Josimar da Silva Rocha $^\mathrm{a^*}$}}\\[2mm]
\end{flushleft}
\footnote{\emph{\hspace*{-0.65cm}{\\{$^\mathrm{a}${\footnote \it  Department of Mathematics, Universidade Tecnol\'ogica Federal do Paran\'a, \\ Corn\'elio Proc\'opio Campus, Corn\'elio Proc\'opio,  Paran\'a,   Brazil.}}}\\
{\it*Corresponding author: E-mail: jrocha@utfpr.edu.br;}}}
{\rule{\linewidth}{1pt}}\vspace*{-0.4cm}
\section*{{\fontsize{14}{1}\selectfont Abstract}}
\footnotesize In quantum mechanics, associative algebras play an important role in understanding symmetries~ and~ operator~ algebras,~ providing~ new algebraic frameworks for describing physical systems. This work classifies associative algebras over a field $K$ that are generated by a finite set $G$ and satisfy a polynomial identity of the form $X^{2} = aX+b,$ where $a$ and $b$ are elements of $K$ and $X$ varies either over all elements of the algebra or over all elements of the multiplicative semigroup $S$ generated by $G.$  One of the results obtained in this work shows that algebras satisfying $X^{2}=0$ over fields of characteristics different from 2 are nilpotent of index  3.\\
\noindent The results obtained were validated computationally using the GAP system.\\[2mm]
\footnotesize{\it{Keywords:} Associative algebra; Polynomial identity; Nilpotency index; Nagata- \quad\indent\quad\indent ~Higman theorem.}\\[2mm]
%\footnotesize{\bf{2010 Mathematics Subject Classification:}} }
\afterpage{
\fancyhead{} \fancyfoot{}
\fancyfoot[C]{\footnotesize\thepage}
\fancyhead[R]{\scriptsize\it {\textbf {}}\\
\\}}
\vspace{-.7 cm}
\section{Introduction}

Associative algebras play a crucial role in algebraic structures and have numerous applications in various fields of mathematics and physics. 
The Irish mathematician, physicist, and astronomer William Hamilton discovered the quaternions in 1843, marking one of the earliest examples of a non-commutative associative algebra \cite{Agore, Huang, Avila}. In a more general framework, Clifford algebras, used in mathematics, physics, engineering, and other fields, extend and generalize the real, complex, and quaternionic analysis. These are associative algebras \cite{Shlaka, Campos}.\\

\noindent Beyond their fundamental importance in algebraic structures, the classification of associative algebras that satisfy polynomial identities has applications in various other areas, as we can see in \cite{Lidl, bratteli, GAP4}. For example, such algebraic structures can contribute to coding theory, offering new ways to construct error-correcting codes with the aim of improving coding efficiency. In Quantum Mechanics, in Physics, these algebras can contribute to the understanding of symmetries and operator algebras to provide new algebraic approaches for physical systems. Furthermore, the classification and computational validation of algebras using the GAP system directly contribute to computational algebra, providing a basis for the development of algorithms and tools for symbolic computation.\\

\noindent When a field $F$ has characteristic 0 and $A$ is an $F$-algebra, we define $A$ as a nil-algebra if there exists a positive integer $n$ such that $a^{n}=0$ for every $a \in A.$ In this case, $n$ is called the nil-index of $A.$ An algebra $A$ is said to be nilpotent of index $m$ if any product of $m$ elements of $A$ is zero.\\

\noindent In \cite{nagata}, Nagata proved that if $A$ is a nil-algebra with nil-index $n,$ then $A$ is nilpotent of index $N(n)$ for some $N(n)\in \mathbb{N}.$ Higman \cite{higman} provided a refined result, showing that $n^{2} \leq N(n)\leq 2^{n}-1.$ Razmyslov \cite{razmyslov} improved Higman's upper bound to $N(n)\geq n^2, $ and Kuzmin \cite{kuzmin} enhanced the lower bound to $N(n)\geq \frac{n(n+1)}{2}.$  Woo Lee \cite{woo_lee} demonstrated that a nil-algebra $A$ with nil-index $3$ is nilpotent of index $6.$ \\

\noindent Subsequent developments expanded the theoretical landscape of PI-algebras. 
In particular, the monograph by \cite{GiambrunoZaicev} provides asymptotic methods for studying polynomial identities, while the edited volume by \cite{BelovRowen} explores the computational aspects of PI-theory, connecting the field with modern computational algebra systems such as \texttt{MAGMA} and \texttt{GAP}. 
Other recent works on algebras satisfying polynomial identities can be found in \cite{Bremner, drensky, Ouedraogo, Grishkov}.\\

\noindent This paper focuses on classifying and characterizing associative algebras generated by a finite set $G = \{ x_{1}, x_{2}, \cdots, x_{m} \}$ over a field $K$ that satisfy a quadratic polynomial identity of the form $X^2 = aX + b, $ where $a, b \in K$ and $X$ can vary over the elements of $A$ or over a multiplicative semigroup $S$ generated by $G$. 

\section{The Case $X^2=k$}
	
The next lemma restricts the case $X^{2} = k$ to the two possibilities $X^{2} = 0$ and $X^{2} = 1$:

\begin{lemma}
	Let $K$ be a field of characteristic $p\neq 2$, and let $S$ be the free multiplicative semigroup generated by the set $\{x_{1}, x_{2}, \cdot, x_{m}\}.$ Let $A$ be an associative $K$-algebra generated by $S,$ and suppose that every element $X \in S$ satisfies   \[ X^{2}= k\] for some $k \in K.$ Then $k\in \{0, 1\}. $
\end{lemma}

\begin{proof} Since $k = (x^2)^2 = k^2$, then $k \in \{0, 1\}$.
\end{proof}

\subsection{The case $X^2 = 0$}

In this subsection, we show that the cases in which every element of the algebra satisfies 
$X^{2}=0$  exhibit different properties depending on whether the characteristic of the algebra is 2 or different from 2, as we shall see in the following theorems.

\begin{theorem}
	Let $K$ be a field of characteristic $p\neq 2$, and let $S$ be the free multiplicative semigroup generated by the set $\{x_{1}, x_{2}, \cdot, x_{m}\}.$ Let $A$ be an associative $K$-algebra generated by $S,$ and suppose that every element $X \in A$ satisfies  \[ X^2 = 0. \] Then $A$ is nilpotent of index at most 3 and  $\displaystyle \dim_{K}(A) \leq  \frac{m(m+1)}{2}. $  
\end{theorem}

\begin{proof}
	Let $x, y, z \in A$, then	
\begin{equation} \label{1x2v}  
(x + y)^{2} = x^{2} + xy + yx + y^{2} = 0 \Rightarrow xy = -yx 
\end{equation}
	
\noindent and

	\begin{equation}\label{2x2v}  (x + yz)^{2} = x^{2} + xyz + yzx + (yz)^{2}  = xyz + yzx = 0 \Rightarrow xyz = - yzx \end{equation}

\noindent Now observe that
	
	\begin{equation} \label{3x2v}  xyz \stackrel{(\ref{1x2v})}{=} -(yxz) \stackrel{(\ref{1x2v})}{=} -(y(-zx)) = yzx \end{equation} 
	
\noindent By (\ref{3x2v}) and (\ref{2x2v}), we obtain that 

	\begin{equation}  2xyz = 0.  \end{equation}

\noindent As $A$ is an Algebra of characteristic $p\neq 2$ and $2xyz=0, $ it follows that $xyz=0.$ Therefore $A$ is nilpotent of index at most 3. \\

\noindent By \eqref{1x2v}, \eqref{2x2v} and \eqref{3x2v}, it  follows that the basis of $A$ is a subset of

	\[ \left(\bigcup_{k=1}^{m} \{ x_{k}\}\right)\bigcup \left(\bigcup_{1\leq i < j \leq m} \{ x_{i}x_{j} \}\right) \]

 \noindent which has $\frac{m(m+1)}{2}$ elements. Hence, the dimension of  $A$ is at most $\frac{m(m+1)}{2}.$ 
	
\end{proof}

\begin{example}\label{e1}
	
Let $K$ be a field of characteristic $p\neq 2$, and let $S$ be the free multiplicative semigroup generated by the set $\{x, y\}.$ Let $A$ be an associative $K$-algebra generated by $S,$ and suppose that every element $X \in A$ satisfies  $X^2 = 0$.
	
\noindent	Thus,
	
	\[ (x + y)^{2} = x^{2} + y^{2} + xy + yx = 0 \Rightarrow xy = -yx \]
	
	\[ (x + xy)^{2} = x^{2} + (xy)^{2} + x^{2}y + xyx  = 0 \Rightarrow xyx = 0 \]

\noindent 	Indeed, we obtain that words of length greater than 2 in the semigroup $S$ generated by $x$ and $y$ are equal to zero.\\
	
	\noindent In this case, any element $w \in A$ can be written as
	
	\[ w = ax + by + cxy, \]
	
\noindent where $a, b, c \in K. $\\

\noindent 	In this case, we can take

	\[ x = \left( \begin{array}{rrrr} 0 & 1 & 0 & 0 \\
		0 & 0 &  0 & 0 \\
		0 & 0 & 0  & 1 \\
		0 & 0 & 0 & 0 \\
	\end{array} \right) \]
	
\noindent and

	\[ y = \left( \begin{array}{rrrr} 0 & 0 & 1 & 0 \\
		0 & 0 & 0 & -1 \\
		0 & 0 & 0  & 0\\
		0 & 0 & 0 & 0 \\
	\end{array}\right), \]
	
\noindent as we can verify by running the code in GAP found in Appendix A.

\end{example}

\begin{theorem}
Let $K$ be a field of characteristic 2, and let $S$ be the free multiplicative semigroup generated by the set $\{x_{1}, x_{2}, \cdot, x_{m}\}.$ Let $A$ be an associative $K$-algebra generated by $S,$ and suppose that every element $X \in A$ satisfies \[ X^2 = 0.\] \\
	
 \noindent Then $A$ is  is nilpotent of index at most $m + 1$ and $\dim_{K}(A) \leq 2^m - 1$.
\end{theorem}

\begin{proof} 
	By \eqref{1x2v}, we obtain that a basis of $A$ over a field $F$ of characteristic 2, satisfying $X^2 = 0$, $\forall X \in A$, is formed by elements of the form $x_1^{\xi_1} x_2^{\xi_2} \dots x_m^{\xi_m}$, where $\xi_{1}, \dots, \xi_{m} \in \{0, 1\}$ and not all $\xi_i$ are zero. Therefore, $\dim_{K}(A) \leq 2^m - 1$.
\end{proof}

\subsection{The case $X^{2} = 1$}
	
The following result shows that an algebra generated by a finitely generated semigroup, in which every element of this semigroup satisfies the equation 
$X^{2}=1$ is abelian, and its dimension depends on the number of generators of the semigroup:	
\begin{theorem}
	Let $K$ be a field, and let $S$ be the free multiplicative semigroup generated by the set $\{x_{1}, x_{2}, \cdot, x_{m}\}.$ Let $A$ be an associative $K$-algebra generated by $S,$ and suppose that every element $X \in S$ satisfies \[ X^2 = 1.\] Then $A$ is an Abelian algebra and $\dim_{K}(A) = 2^m$.
\end{theorem}

\begin{proof} First, observe that  $yx = x^{2}yxy^{2} = x(xy)^{2}y = xy, \forall x, y \in \{x_{1}, \cdots, x_{m}\}.$ \\
	
	\noindent Therefore, the generators of $A$ are of the form $x_{1}^{\xi_{1}} x_{2}^{\xi_{2}} \cdots x_{m}^{\xi_{m}}, $ where $\xi_{1}, \cdots, \xi_{m} \in \{0, 1\}.$ \\
	
	\noindent Thus, $\dim_{K}(A) = 2^{m}.$ 

\end{proof}

\begin{example}\label{e2}
	
  Let $K$ be a field,  and let $S$ be the free multiplicative semigroup generated by the set $\{x, y\}.$ Let $A$ be an associative $K$-algebra generated by $S,$ and suppose that every element $X \in S$ satisfies \[ X^2 = 1.\] Then there exists a basis $\{1, x, y, xy\}$ where any element $w \in A$ can be written as

	\[ w = \alpha_{1} + \alpha_{2}x + \alpha_{3}y + \alpha_{4}xy, \]
	
\noindent where $\alpha_{1}, \alpha_{2},  \alpha_{3},  \alpha_{4} \in K.$

	\noindent If $w_{1} = \alpha_{1} + \alpha_{2}x + \alpha_{3} y + \alpha_{4}xy $ and $w_{2} = \beta_{1} + \beta_{2}x + \beta_{3}y + \beta_{4}xy, $ then
	
	\[ \begin{array}{lll} w_{1}w_{2}   & = & (\alpha_{1}\beta_{1} + \alpha_{2}\beta_{2} + \alpha_{3}\beta_{3} + \alpha_{4}\beta_{4}) \\
		& + & (\alpha_{1}\beta_{2} + \alpha_{2}\beta_{1} + \alpha_{3}\beta_{4} + \alpha_{4}\beta_{3})x \\
		& + & (\alpha_{1}\beta_{3} + \alpha_{2}\beta_{4} + \alpha_{3}\beta_{1} + \alpha_{4}\beta_{2})y \\
		& + & (\alpha_{1}\beta_{4} + \alpha_{2}\beta_{3} + \alpha_{3}\beta_{2} + \alpha_{4} \beta_{1})xy  \end{array} \]

\noindent In this case, we can take
	
	\[ x = \left(\begin{array}{rrrr} 0 & 1 & 0 & 0 \\
		1 & 0 & 0 & 0 \\
		0 & 0 & 0 & 1 \\
		0 & 0 & 1 & 0 \\
	\end{array} \right) \]
	
\noindent and
	\[ y = \left( \begin{array}{rrrr}
		0 & 0 & 1 & 0 \\
		0 & 0 & 0 & 1 \\
		1 & 0 & 0 & 0 \\
		0 & 1 & 0 & 0 \\
	\end{array} \right), \]
	
\noindent as we can verify by running the GAP code from Appendix B.
\end{example}

\section{The Case $X^{2} = kX$}

The following lemma restricts the case 

$X^{2} = kX$ to the two subcases 
$X^{2} = 0$ and $X^{2} = X,$  where the case 
$X^{2} = 0$ has already been analyzed in Subsection 2.1:
	
\begin{lemma}  Let $K$ be a field, and let $S$ be the free multiplicative semigroup generated by the set $\{x_{1}, x_{2}, \cdot, x_{m}\}.$ Let $A$ be an associative $K$-algebra generated by $S,$ and suppose that every element $X \in S$ satisfies \[ X^2 = kX\] for some $k\in K.$  Then $k \in \{0, 1\}$.
\end{lemma}

\begin{proof} Since $X^{2} = kX$ for every $X \in S$, it follows that
\begin{equation}\label{eq:first}
(x^{2})^{2} = kx^{2} = k^{2}x.
\end{equation}
On the other hand, we also have
\begin{equation}\label{eq:second}
(x^{2})^{2} = (kx)^{2} = k^{2}x^{2} = k^{3}x.
\end{equation}

\noindent Comparing equations~(\ref{eq:first}) and~(\ref{eq:second}), we obtain
\begin{equation}\label{eq:third}
k^{2}x = k^{3}x,
\end{equation}

\noindent which implies that

\begin{equation}\label{eq:fourth}
k^{2}(k - 1)x = 0.
\end{equation}

\noindent Therefore, $k \in \{0, 1\}.$
\end{proof}

\subsection{The  case $X^{2} = X$}

In this subsection, we classify algebras in two distinct cases: when every element of the algebra satisfies the equation \(X^{2} = X\), and when every element of the generating semigroup of the algebra satisfies this equation, as we shall see in the following theorems.
	
\begin{theorem}\label{x2=x} Let $K$ be a field, and let $S$ be the free multiplicative semigroup generated by the set $\{x, y\}.$ Let $A$ be an associative $K$-algebra generated by $S,$ and suppose that every element $X \in A$ satisfies
	\[
	X^2 = X.
	\]
	
\noindent  Then:

	\begin{itemize}
		\item[(i)] $A$ is abelian;
		\item[(ii)] $A$ has characteristic $2$;
		\item[(iii)] $\dim_K(A) \leq 3.$
	\end{itemize}
\end{theorem}

\begin{proof} 
	Let $w \in A$. Thus, $(2w)^2 = 4w$ and $(2w)^2 = 2w$, which implies $2w = 0$. Therefore, $A$ has characteristic 2.

	\noindent Since

	\begin{equation}
		(x + y)^{2} = x + y
	\end{equation}

\noindent and
	
	\begin{equation}
		(x + y)^{2} = x^{2} + y^{2} + xy + yx = x + y + xy + yx
	\end{equation}
	
\noindent then

	\begin{equation} \label{eqcom} 
		xy = yx.
	\end{equation}

\noindent Therefore, $A$ is an Abelian algebra.\\
	
	\noindent By \eqref{eqcom} and by definition of $A,$ it follows that a basis for $A$ is a subset of  $\{x, y, xy\}.$ Thus, $\dim_{K}(A)\leq 3.$
	
\end{proof}

\begin{example} \label{e3} The  associative  $\mathbb{Z}_{2}$-algebra  generated by $x$ and $y$, where $x = \left[ \begin{array}{rrr} 1 & 0 & 0 \\
		0 & 0 & 1 \\ 0 & 0 & 1 \end{array} \right] $ and $y = \left[ \begin{array}{rrr} 0 & 0 & 1 \\ 0 & 1 & 0 \\ 0 & 0 & 1 \end{array} \right]$ \\
	
\noindent satisfies the hypotheses of Theorem \ref{x2=x} and has  dimension $3$ as a vector space, as we verify by running the GAP code from Appendix C. 

\end{example}

\begin{theorem}\label{t32} Let $K$ be a field, and let $S$ be the free multiplicative semigroup generated by the set $\{x_{1}, x_{2}, \cdot, x_{m}\}.$ Let $A$ be an associative $K$-algebra generated by $S,$ and suppose that every element $X \in A$ satisfies \[ X^2 = X.\] Then:
	
	\begin{itemize}
		\item $\mathcal{A}$ is abelian;
		\item $A$ has characteristic 2;
		\item $\dim_{K}(A) \leq  2^{m}-1;$ 
	\end{itemize}
\end{theorem}

\begin{proof} By  Theorem \ref{x2=x}, it follows that $A$ is Abelian and $A$ has characteristic 2.
	As $A$ is Abelian and by definition of $A$, it follows that a basis of $A$ is a subset of \[  \{  x_{1}^{\alpha_1} x_{2}^{\alpha_2} \dots x_{m}^{\alpha_m}\mid  \alpha_{1}, \dots, \alpha_{m} \in \{0, 1\}, \alpha_{1}^{2} + \cdots + \alpha_{m}^{2} \neq 0 \} \] Thus, $A$ has dimension at most $2^m - 1$.
\end{proof}

\begin{example} \label{e4}The  associative  $\mathbb{Z}_{2}$-algebra  generated by $S = \langle x,  y,z\rangle,$ where 
\[ x = \left[\begin{array}{rrrrrrr} 1 & 0 & 0 & 0 & 0 & 0 & 0\\
0& 0& 0 & 1 & 0 & 0 & 0 \\
0 & 0 & 0 & 0 & 1 & 0 & 0 \\
0 & 0 & 0 & 1 & 0 & 0 & 0 \\
0& 0 & 0 & 0 & 1 & 0& 0 \\
0 & 0 & 0 & 0 & 0 & 0 & 1\\
0 & 0 & 0& 0 & 0 & 0 & 1\\
\end{array} \right],  	y=\left[ \begin{array}{rrrrrrr} 
0& 0 & 0 & 1 & 0 & 0 & 0\\
0 & 1 & 0 & 0 & 0 & 0 & 0\\
0& 0& 0& 0 & 0 & 1 & 0\\
0 & 0 & 0 & 1 & 0 & 0 & 0\\
0 & 0 & 0 & 0 & 0 & 0& 1\\
0 & 0 & 0 & 0 & 0 & 1 & 0\\
0 & 0 & 0 & 0 & 0 & 0 & 1\\
\end{array}\right]\] e 
$	z =\left[ \begin{array}{rrrrrrr} 
0 & 0 & 0 & 0 & 1 & 0 & 0\\
0 & 0 & 0 & 0 & 0 & 1 & 0\\
0 & 0 & 1 & 0 & 0 & 0 & 0\\
0 & 0 & 0 & 0 & 0 & 0 & 1\\
0 & 0 & 0& 0 & 1 & 0 & 0\\
0 & 0 & 0 & 0 & 0 & 1 & 0\\
0 & 0 & 0 & 0 & 0 & 0 & 1\\
\end{array}\right]$ 

\noindent  satisfies the hypotheses of Theorem \ref{t32} and has  dimension $7$ as a vector space, as we verify by running the GAP code from Appendix D. 
\end{example}

\begin{theorem}\label{t33} Let $K$ be a field, and let $S$ be the free multiplicative semigroup generated by the set $\{x, y\}.$ Let $A$ be an associative $K$-algebra generated by $S,$ and suppose that every element $X \in S$ satisfies \[ X^2 = X.\] 
	Then  $\dim_{K}(A) \leq 6$.
\end{theorem}

\begin{proof}  It suffices to observe that the set $\{x, y, xy, yx, xyx, yxy\}$ generates the algebra $A$. Note that if $A = \langle x, y \rangle$, where
	
	\[ x = \left[ \begin{array}{lllllll}
		0 & 1 & 0 & 0 & 0 & 0 & 0 \\
		0 & 1 & 0 & 0 & 0 & 0 & 0 \\
		0 & 0 & 0 & 1 & 0 & 0 & 0 \\
		0 & 0 & 0 & 1 & 0 & 0 & 0 \\
		0 & 0 & 0 & 0 & 0 & 1 & 0 \\
		0 & 0 & 0 & 0 & 0 & 1 & 0 \\
		0 & 0 & 0 & 1 & 0 & 0 & 0 
	\end{array} \right] \]
	and
	\[ y =\left[ \begin{array}{lllllll}
		0 & 0 & 1 & 0 & 0 & 0 & 0 \\
		0 & 0 & 0 & 0 & 1 & 0 & 0 \\
		0 & 0  & 1  & 0 & 0 & 0 & 0 \\
		0 & 0 & 0 & 0 & 0 & 0 & 1 \\
		0 & 0 & 0 & 0 & 1 & 0 & 0 \\
		0 & 0 & 0 & 0 & 1 & 0 & 0 \\
		0 & 0 & 0 & 0 & 0 & 0 & 1 \\
	\end{array} \right], \]
	then  $A$ satisfies the conditions of the proposition, and $\beta = \{x, y, xy, yx, xyx, yxy\}$ is a basis for $A$ with $\dim_{K}(A) = 6$, as  we can verify by running the GAP code from Appendix E.
\end{proof}

\section{The Case $X^{2} = aX + b,$  with $a, b \in K^{\ast}$}

In this section, we restrict our attention  to the  associative algebras over $K$ that satisfies the equation quadratic equation
\[ X^{2} = aX+b,\]  where
$X$ ranges over the elements of a semigroup generated by a set $\{x_{1},\cdots x_{m}\}$ and $a, b \in K^{\ast}.$

\begin{theorem} Let $K$ be a field with characteristic $2$, and let $S$ be the free multiplicative semigroup generated by the set $\{x, y\}.$ Let $A$ be an associative $K$-algebra generated by $S,$ and suppose that every element $X \in S$ satisfies  \[ X^2 = X + 1. \] Then $A$ is the field $GF(4)$ with four elements $\{0, 1, x, 1 + x\}$.
\end{theorem}

\begin{proof} We have
	\begin{equation} \label{comp1}
		(x^{2}y)^{2} = x^{2}y + 1 = (x + 1)y + 1 = xy + y + 1
	\end{equation}
	and
	\begin{equation} \label{comp2}
		\begin{array}{ll}
			(x^{2}y)^{2} & = ((x + 1)y)^{2} \\
			& = (xy + y)^{2} \\
			& = (xy)^{2} + y^{2} + xy^{2} + yxy \\
			& = xy + 1 + y + 1 + x(y + 1) + yxy \\
			& = 2xy + 2 + y +x +  yxy = y + x  + yxy\\
 \end{array}
\end{equation}
	By \eqref{comp1} and \eqref{comp2}, we have
	
	\begin{equation} \label{comp3}
		yxy = xy  + x + 1
	\end{equation}
	
	\begin{equation}\label{comp4}
		y^{2}xy = yxy + yx + y = xy + x + 1 + yx + y 
	\end{equation}
	
	\begin{equation}\label{comp5}
		y^{2}xy = (y + 1)xy = yxy + xy = 2xy + x + 1 = x + 1
	\end{equation}
	By (\ref{comp4}) and (\ref{comp5}), we have
	\begin{equation} \label{comp6}
		xy = yx + y
	\end{equation}

\noindent  Thus, by symmetry,

	\begin{equation}\label{comp7}
		yx = xy + x 
	\end{equation}

\noindent  By (\ref{comp6}) and (\ref{comp7}), we have that $x = y.$ 

\end{proof}

\begin{theorem}\label{t42} Let $K$ be a field with characteristic $p\neq 2,$ and let $S$ be the free multiplicative semigroup generated by the set $\{x, y\}.$ Let $A$ be an associative $K$-algebra generated by $S,$ and suppose that every element $X \in S$ satisfies  \[ X^2 = aX + b\]  with $a, b \in K^*$. Then
	
	\begin{itemize}
		\item[(i)] $a=2 $ and $b = -1;$
		\item[(ii)] $\{ 1, x, y, xy\} $ is a basis for $A$ as a vector space over $K$; 
		\item[(iii)] $\dim_{K}(A) = 4.$ 
	\end{itemize}
	
\end{theorem}

\begin{proof}
	 Since $X^2 = aX + b$, then
	
	\begin{equation} \label{comp8}
		\begin{array}{ll} 
			(X^{2})^{2} & = (aX+b)^{2} \\
			&  = a^{2}X^{2} + 2abX + b^{2} \\
			& = a^{2}(aX + b) + 2abX + b^{2} \\
			& = (a^{3} + 2ab)X + a^{2}b + b^{2} \\
		\end{array}
	\end{equation}
	and
	\begin{equation} \label{comp9}
		(X^{2})^{2} = aX^{2} + b = a(aX + b) + b = a^{2}X + ab + b 
	\end{equation}

\noindent  Therefore, we obtain

\begin{tiny}
	\begin{equation} \label{comp10}\left\{
		\begin{array}{l}
			a^{2} = a^{3} + 2ab \\
			ab + b = a^{2}b + b^{2} \\
		\end{array} \right. \Rightarrow  \left\{
		\begin{array}{l}
			a = a^{2} + 2b \\
			a + 1 = a^{2} + b \\
		\end{array} \right. \Rightarrow \left\{ \begin{array}{l} b = -1 \\
			a^{2} -a - 2 = 0 \end{array}\right.
		\Rightarrow \left\{ 
		\begin{array}{l}
			b = -1 \\
			a = 2 \text{ or } a = -1 \end{array} \right. 
	\end{equation}
	
	\vspace{1cm}
\end{tiny}
	
 For $X^2 = aX + b = -X - 1$, we have
	\begin{equation} \label{comp11}
		\begin{array}{ll}
			(x^{2}y)^{2}  & = -x^{2}y - 1 \\
			& = - (-x - 1)y - 1 \\
			&  = xy + y - 1 \end{array} 
	\end{equation}
	and
	\begin{equation} \label{comp12}
		\begin{array}{ll}
			(x^{2}y)^{2} & = [(-x -1)y]^{2}  \\
			& = (xy +y)^{2} \\
			& = (xy)^{2} + y^{2} +xy^{2} + yxy \\
			& = -xy - 1 - y - 1 +x(-y - 1) + yxy \\
			& = -2xy -2 - y - x + yxy \end{array}
	\end{equation}

\noindent  By (\ref{comp11}) and (\ref{comp12}), we obtain 

	\begin{equation} \label{comp13}
		yxy = 3xy + 2y + x + 1 
	\end{equation}

\noindent  Now,

	\begin{equation} \label{comp14}
		\begin{array}{ll}
			xyxy & = 3x^{2}y + 2xy + x^{2} + x \\
			& = 3(-x-1)y + 2xy +(-x - 1) + x \\
			& = -3xy -3y + 2xy - 1 = -xy - 3y - 1 
		\end{array}
	\end{equation}
	
\noindent  and

	\begin{equation} \label{comp15}
		\begin{array}{ll}
			xyxy = -xy - 1 
		\end{array}
	\end{equation}
	
\noindent  By (\ref{comp14}) and (\ref{comp15}),

	\begin{equation}
		3y = 0 
	\end{equation}

\noindent  	By symmetry, we can prove that $3x = 0.$

	\noindent  Therefore, $K$ has characteristic 3 and $a=-1 = 2.$
	
	\noindent  Hence, we have proved (i), since in all cases we can conclude that 
	$a=2$ and $b=-1.$

\noindent  Since $K$ has characteristic 3, we obtain that
	
	\begin{equation} \label{comp16}
		yxy =  2y + x + 1 
	\end{equation}

\noindent  Thus,

	\begin{equation}\label{comp17}
		y^{2}xy = y(2y + x + 1) = 2y^{2} + yx + y = -2y - 2 + yx + y = -2 + yx +2y 
	\end{equation}

\noindent  and

	\begin{equation}\label{comp18}
		y^{2}xy = (- y - 1)xy = -yxy - xy = -(2y + x + 1) - xy = -2y - x - 1 - xy
	\end{equation}

\noindent  By (\ref{comp17}) and (\ref{comp18}), we obtain that 

	\begin{equation} \label{comp19}
		yx + xy = -4y - x + 1 = 2y + 2x + 1
	\end{equation}
	
\noindent  	For $X^{2} = aX + b = 2X - 1,$ we have  
	
	\begin{equation} \label{comp20}
		\begin{array}{ll}
			(x^{2}y)^{2}  & = 2x^{2}y - 1 \\
			& = 2(2x - 1)y - 1 \\
			&  = 4xy -2y - 1 \end{array} 
	\end{equation}
	
\noindent  and

	\begin{equation} \label{comp21}
		\begin{array}{ll}
			(x^{2}y)^{2} & = [(2x -1)y]^{2}  \\
			& = (2xy -y)^{2} \\
			& = 4(xy)^{2} + y^{2} -2xy^{2} -2yxy \\
			& = 8xy - 4 +2y - 1 -2x(2y - 1) -2yxy \\
			& = 4xy -5 +2y +2x - 2yxy \end{array}
	\end{equation}
	
	\noindent  By (\ref{comp20}) and (\ref{comp21}), we obtain that 

	\begin{equation} \label{comp22}
		yxy = 2y + x - 2 
	\end{equation}

\noindent  	Now,

	\begin{equation} \label{comp23}
		\begin{array}{ll}
			y^{2}xy & = 2y^{2} + yx - 2y \\
			& = 4y - 2 + yx - 2y   \\
			& = 2y - 2 + yx  
		\end{array}
	\end{equation}
	
\noindent  and

	\begin{equation} \label{comp24}
		\begin{array}{ll}
			y^{2}xy & = (2y - 1)xy  \\
			& = 2yxy - xy \\
			& = 2(2y + x - 2) - xy \\
			& = 4y + 2x - 4 - xy   
		\end{array}
	\end{equation}
	
\noindent  By (\ref{comp23}) and (\ref{comp24}),

	\begin{equation}
		yx + xy = 2y + 2x - 2  
	\end{equation}

	\noindent  The equations from \eqref{comp11} to \eqref{comp24} show that $\{1, x, y, xy\} $
	 forms a basis of $A$ as a vector space over 
	$K$; consequently, 
	$\dim_{K}(A)=4$, thereby proving (ii) and (iii).

 \noindent  	Note that if $A$ is an algebra satisfying the above conditions, then every element  $w \in A$ can be written as
	\[ w = \alpha_{1} + \alpha_{2}x + \alpha_{3}y + \alpha_{4}xy, \]
	
\noindent  where $\alpha_{1}, \alpha_{2},  \alpha_{3},  \alpha_{4} \in K.$

	\noindent  Moreover, if  $w_{1} = \alpha_{1} + \alpha_{2}x + \alpha_{3} y + \alpha_{4}xy $ and $w_{2} = \beta_{1} + \beta_{2}x + \beta_{3}y + \beta_{4}xy, $ then

	\[ \begin{array}{lll} w_{1}w_{2}   & = & (\alpha_{1}\beta_{1} - \alpha_{2}\beta_{2} - 2\alpha_{3}\beta_{2} - \alpha_{3}\beta_{3} - 2\alpha_{3}\beta_{4} - 2\alpha_{4}\beta_{2} - \alpha_{4}\beta_{4}) \\
		& + & (\alpha_{1}\beta_{2} + \alpha_{2}\beta_{1} + 2\alpha_{3}\beta_{2} + 2\alpha_{2}\beta_{2} + \alpha_{3}\beta_{4} + 2\alpha_{4}\beta_{2} - \alpha_{4}\beta_{3})x \\
		& + & (\alpha_{1}\beta_{3} - \alpha_{2}\beta_{4} + 2\alpha_{3}\beta_{2}  + \alpha_{3}\beta_{1}+ 2\alpha_{3}\beta_{3} + 2\alpha_{3}\beta_{4} + \alpha_{4}\beta_{2})y \\
		& + & (\alpha_{1}\beta_{4} + \alpha_{2}\beta_{3} + 2\alpha_{2}\beta_{4}  - \alpha_{3}\beta_{2} + \alpha_{4}\beta_{1} + 2\alpha_{4}\beta_{3} + 2\alpha_{4}\beta_{4})xy
	\end{array} \]
	
\noindent  	In this case, we can take
	
	\[ x = \left(\begin{array}{rrrr} 0 & 1 & 0 & 0 \\
		-1 & 2 & 0 & 0 \\
		0 & 0 & 0 & 1 \\
		0 & 0 & -1 & 2 \\
	\end{array} \right) \]
	
\noindent  and

	\[ y = \left( \begin{array}{rrrr}
		0 & 0 & 1 & 0 \\
		-2 & 2 & 2 & -1 \\
		-1 & 0 & 2 & 0 \\
		-2 & 1 & 2 & 0 \\
	\end{array} \right), \]

\noindent  	as we can verify by running the GAP code from Appendix F.

\end{proof}

\begin{theorem}\label{t43} Let $K$ be a field with characteristic $p\neq 2,$ and let $S$ be the free multiplicative semigroup generated by the set $\{x_{1},\cdots, x_{m}\}.$ Let $A$ be an associative $K$-algebra generated by $S,$ and suppose that every element $X \in S$ satisfies  \[ X^2 = aX + b\]  with $a, b \in K^*$. Then

	\begin{itemize}
		\item[(i)] $a=2 $ and $b = -1;$
		\item[(ii)] $\displaystyle \{1\} \bigcup \{x_{1}, \cdots, x_{m}\} \bigcup_{1\leq j<k \leq m}\{ x_{j}x_{k}\}$ is a basis for $A$ as a vector space over $K$;
		\item[(iii)] $\displaystyle \dim_{K}(A) = \frac{m^{2}+m+2}{2}.$
	\end{itemize}

\end{theorem}

\begin{proof} (i) follows from  Theorem \ref{t42}. Moreover, it follows from the proof of the Theorem  \ref{t42} that for every
$x,y,z \in S,$ the following equations hold:

\begin{equation}\label{t43-1}
yz=2y + 2z - 2 - zy
\end{equation}

\begin{equation}\label{t43-2}
 xz=2x + 2y - 2 - zx
\end{equation}

\begin{equation}\label{t43-3}
 zy = 2z + 2y - 2 -yz
\end{equation}

\begin{equation}\label{t43-4}
xyz = (xy)z = 2xy + 2z - 2 - zxy
\end{equation}

\noindent Thus

\begin{equation}\label{t43-5}
\begin{array}{cl}
xyz & = x(yz) \stackrel{\eqref{t43-1}}{=} x(2y + 2z - 2 - zy) \\
& = 2xy + 2xz - 2x - xzy \\
& \stackrel{\eqref{t43-2}}{=} 2xy + 2xz - 2x -(2x + 2z - 2 - zx)y \\
& = 2xy + 2xz - 2x -2xy -2zy + 2y + zxy \\
& = -2x + 2y + 2xz - 2zy + zxy \\
& \stackrel{\eqref{t43-3}}{=} -2x + 2y + 2xz - 2(2z + 2y - 2 - yz) + zxy \\
& = 4 -2x -2y -4z + 2xz +2yz \\
\end{array}
\end{equation}

\noindent By \eqref{t43-4} and \eqref{t43-5},  we have

\begin{equation}
2xyz = 2 - 2x - 2y - 2z + 2xy + 2xz + 2yx
\end{equation}

\noindent Therefore, since the characteristic of
$A$ is different from 2, it follows that

\begin{equation}\label{t43-6}
xyz = 1 - x - y - z + xy + xz + yx
\end{equation}

\noindent Hence, (ii) follows from the definition of the algebra $A$ , from equations \eqref{t43-1}–\eqref{t43-6}, and from (i). Moreover, by (ii) it follows that \[ \dim_{K}(A) = 1 + \binom{m+1}{2} = 1 + \frac{(m+1)m}{2} = \frac{2 + (m+1)m}{2} = \frac{m^2+m+2}{2}  \]

\noindent or

\[ \dim_{K}(A) = 1 + m + \binom{m}{2} = 1 + m + \frac{m(m-1)}{2} = \frac{2 + 2m + m^{2} - m}{2} = \frac{m^2+m+2}{2} \]

\end{proof}

\begin{example}
The associative $GF(3)$-algebra generated by $S = \langle x, y, z \rangle,$ where
\[
x = \left[
\begin{array}{rrrrrrr}
0 & 1 & 0 & 0 & 0 & 0 & 0 \\
-1 & 2 & 0 & 0 & 0 & 0 & 0 \\
0 & 0 & 0 & 0 & 1 & 0 & 0 \\
0 & 0 & 0 & 0 & 0 & 1 & 0 \\
0 & 0 & -1 & 0 & 2 & 0 & 0 \\
0 & 0 & 0 & -1 & 0 & 2 & 0 \\
1 & -1 & -1 & -1 & 1 & 1 & 1
\end{array}
\right],
\quad
y = \left[
\begin{array}{rrrrrrr}
0 & 0 & 1 & 0 & 0 & 0 & 0 \\
-2 & 2 & 2 & 0 & -1 & 0 & 0 \\
-1 & 0 & 2 & 0 & 0 & 0 & 0 \\
0 & 0 & 0 & 0 & 0 & 0 & 1 \\
-2 & 1 & 2 & 0 & 0 & 0 & 0 \\
-1 & 1 & 1 & -1 & -1 & 1 & 1 \\
0 & 0 & 0 & -1 & 0 & 0 & 2
\end{array}
\right],
\]
\\

\noindent and
\[
z = \left[
\begin{array}{rrrrrrr}
0 & 0 & 0 & 1 & 0 & 0 & 0 \\
-2 & 2 & 0 & 2 & 0 & -1 & 0 \\
-2 & 0 & 2 & 2 & 0 & 0 & -1 \\
-1 & 0 & 0 & 2 & 0 & 0 & 0 \\
-3 & 1 & 1 & 3 & 1 & -1 & -1 \\
-2 & 1 & 0 & 2 & 0 & 0 & 0 \\
-2 & 0 & 1 & 2 & 0 & 0 & 0
\end{array}
\right].
\]

\noindent satisfies the hypotheses of Theorem \ref{t43} and has dimension $7$ as a vector space, as we verify by running the GAP code from Appendix G.

\end{example}

\section{Conclusion}

In this work, we  classify associative algebras that are generated by a finite set of elements and satisfy specific quadratic polynomial identities. The results presented not only generalize  existing contributions to the theory of  nil-algebras but also provide new insights into the structure and dimensional characteristics  of associative algebras under various conditions. \\
	
\noindent As future work, this study may be extended to include associative algebras satisfying polynomial equations of degree greater than two.\\

\noindent Computational validation was performed using the {\bf GAP } software package, which confirmed the theoretical results.\\[3mm]

\noindent{\Large\bf  Disclaimer (Artificial Intelligence)}\\\

\noindent Author(s) hereby declare that ~NO ~generative~ AI~ technologies ~such ~as ~Large Language Models (ChatGPT, COPILOT, etc) and text-to-image generators have been used during writing or editing of manuscripts. 

\section*{Acknowledgements}

The author is grateful to the referees for their careful review and valuable comments and remarks to improve this manuscript.\\

\noindent{\Large\bf  Competing Interests}\\\\
Author has declared that no competing interests exist.

\nocite{*}

\printbibliography

\newpage
\section*{Appendix}

\appendix

\section{Verification Code for Example \ref{e1}}

\begin{verbatim}
	# Define the finite field GF(3)
	K := GF(3);
	
	# Left regular representation of the two generators of the algebra 
	# as two 4×4 matrices over GL(3)
	x := [[0,1,0,0],[0,0,0,0],[0,0,0,1],[0,0,0,0]] * One(K);
	y := [[0,0,1,0],[0,0,0,-1],[0,0,0,0],[0,0,0,0]] * One(K);
	
	# Construct the algebra A generated by x and y
	A := Algebra(K, [x, y]);
	
	
	# Check if A is nilpotent of index 3
	v := true;
	for i in A do
	for j in A do
	for k in A do
	if not IsZero(i*j*k) then
	v := false; break;
	fi;
	od;
	if v = false then break; fi;
	od;
	if v = false then break; fi;
	od;
	
	if v = false then
	Print("This algebra isn't nilpotent of index 3!");
	else
	Print("This algebra is nilpotent of index 3!");
	fi;
	
	# Check if every element in A satisfies x^2 = 0
	v := true;
	for i in A do
	if not IsZero(i*i) then
	v := false; break;
	fi;
	od;
	
	if v = false then
	Print("\nThis algebra doesn't satisfy the equation X^2=0");
	else
	Print("\nThis algebra satisfy the given condition!\n");
	fi;
	
	
	# Get a basis B for the algebra A
	B := Basis(A);
	
	# Compute the number of elements and dimension of A
	s := Size(A);
	t := Size(B);
	
	# Print size and dimension
	Print("\nSize of  A:", s, "\n");
	Print("\nDimension of A as a vectorial space:", t, "\n");
	
	#Creates the vector subspace over K generated by {x, y, x*y}
	Sub := VectorSpace(K, [x, y, x*y]);
	
	#Computes the dimension of the subspace Sub
	dim := Dimension(Sub);
	Print("\n Dimension of the subspace generated by {x, y, xy}: ",
 dim, "\n");\end{verbatim}

\section{Verification Code for Example \ref{e2}}

\begin{verbatim}
	# Define the finite field GF(5)
	K:=GF(5);
	
	# Left regular representation of the two generators of the algebra 
	# as two 4x4 matrices over GL(5)
	x:=[[0,1,0,0],[1,0,0,0],[0,0,0,1],[0,0,1,0]]*One(K);
	y:=[[0,0,1,0],[0,0,0,1],[1,0,0,0],[0,1,0,0]]*One(K);
	
	#Construct the semigroup S generated by x and y
	S:=Semigroup(x,y);
	
	
	# Construct the algebra A generated by x and y
	A := Algebra(K, [x, y]);
	# Check if every element in S satisfies X^2 = 1
	
	v:=true;
	for i in S do
	if not IsZero(i^2-One(A)) then v:=false; break; fi;
	od;
	
	if v=false then
	Print("\nThis algebra doesn't satisfy the given condition!\n");
	else
	Print("\nThis algebra satisfy the given condition!\n");
	fi;
	
	# Get a basis B for the algebra A
	B := Basis(A);
	
	# Compute the number of elements and dimension of A
	s := Size(A);
	t := Size(B);
	
	# Print size and dimension
	Print("\nSize of  A:", s, "\n");
	Print("\nDimension of A as a vectorial space:", t, "\n");
	
	#Creates the vector subspace over K generated by {1, x, y, xy}
	Sub := VectorSpace(K, [One(A), x, y, x*y]);
	
	#Computes the dimension of the subspace Sub
	dim := Dimension(Sub);
	
	Print("\n Dimension of the subspace generated by {1, x, y, xy}: ",
 dim, "\n");
\end{verbatim}

\section{Verification Code for Example \ref{e3}}

\begin{verbatim}
	# Define the finite field GF(2)
	K:=GF(2);`


	
	# Left regular representation of the two generators of the algebra 
	# as two 3x3 matrices over GL(2)
	x:=[[1,0,0],[0,0,1],[0,0,1]]*One(K);
	y:=[[0,0,1],[0,1,0],[0,0,1]]*One(K);
	
	
	# Construct the algebra A generated by x and y
	A := Algebra(K, [x, y]);
	
	
	# Check if A is abelian
	if (IsAbelian(A)=true) then Print("\nA is Abelian");
	else Print("\nA is not Abelian"); fi;
	
	# Check if every element in S satisfies X^2 = X
	v:=true;
	for i in A do
	if not IsZero(i^2-i) then v:=false; break; fi;
	od;
	
	if v=false then
	Print("\nThis algebra doesn't satisfy the given condition!\n");
	else
	Print("\nThis algebra satisfy the given condition!\n");
	fi;
	
	# Get a basis B for the algebra A
	B := Basis(A);
	
	# Compute the number of elements and dimension of A
	s := Size(A);
	t := Size(B);
	
	# Print size and dimension
	Print("\nSize of  A:", s, "\n");
	Print("\nDimension of A as a vectorial space:", t, "\n");
	
	#Creates the vector subspace over K generated by {x, y, xy}
	Sub := VectorSpace(K, [x, y, x*y]);
	
	#Computes the dimension of the subspace Sub
	dim := Dimension(Sub);
	
	Print("\n Dimension of the subspace generated by {x, y, xy}: ", dim, "\n");
		\end{verbatim}

\section{Verification Code for Example \ref{e4}}

\begin{verbatim}
	# Define the finite field GF(2)
	K:=GF(2);
	
	# Left regular representation of the three generators of the algebra 
	# as two 7x7 matrices over GL(2)
	x:=[[1,0,0,0,0,0,0],[0,0,0,1,0,0,0],[0,0,0,0,1,0,0],
	[0,0,0,1,0,0,0],[0,0,0,0,1,0,0],[0,0,0,0,0,0,1],[0,0,0,0,0,0,1]
	]*One(K);
	y:=[[0,0,0,1,0,0,0],[0,1,0,0,0,0,0],[0,0,0,0,0,1,0],
	[0,0,0,1,0,0,0],[0,0,0,0,0,0,1],[0,0,0,0,0,1,0],
	[0,0,0,0,0,0,1]]*One(K);
	z:=[[0,0,0,0,1,0,0],[0,0,0,0,0,1,0],[0,0,1,0,0,0,0],[0,0,0,0,0,0,1],
	[0,0,0,0,1,0,0],[0,0,0,0,0,1,0],[0,0,0,0,0,0,1]]*One(K);
	
	
	# Construct the algebra A generated by x, y and z
	A := Algebra(K, [x, y, z]);
	
	
	# Check if A is abelian
	if (IsAbelian(A)=true) then Print("\nA is Abelian");
	else Print("\nA is not Abelian"); fi;
	
	# Check if every element in A satisfies X^2 = X
	v:=true;
	for i in A do
	if not IsZero(i^2-i) then v:=false; break; fi;
	od;
	
	if v=false then
	Print("\nThis algebra doesn't satisfy the given condition!\n");
	else
	Print("\nThis algebra satisfy the given condition!\n");
	fi;
	
	# Get a basis B for the algebra A
	B := Basis(A);
	
	# Compute the number of elements and dimension of A
	s := Size(A);
	t := Size(B);
	
	# Print size and dimension
	Print("\nSize of  A:", s, "\n");
	Print("\nDimension of A as a vectorial space:", t, "\n");
	
	#Creates the vector subspace over K generated by {x, y, z, xy, xz, 
 yz, xyz}
	
	Sub:=VectorSpace(K,[x,y,z,x*y,x*z,y*z,x*y*z]);
	
	#Computes the dimension of the subspace Sub
	dim := Dimension(Sub);
	
	Print("\n Dimension of the subspace generated by {x, y, z, xy, xz, yz,
 xyz}: ", dim, "\n");
\end{verbatim}

\section{Verification Code for Theorem \ref{t33}}

\begin{verbatim}
	# Define the finite field GF(5)
	K:=GF(5);
	
	# Left regular representation of the two generators of the algebra 
	# as two 7x7 matrices over GL(5)
	x:=[[0,1,0,0,0,0,0],[0,1,0,0,0,0,0],[0,0,0,1,0,0,0],[0,0,0,1,0,0,0],
	[0,0,0,0,0,1,0],[0,0,0,0,0,1,0],[0,0,0,1,0,0,0]]*One(K);
	y:=[[0,0,1,0,0,0,0],[0,0,0,0,1,0,0],[0,0,1,0,0,0,0],
	[0,0,0,0,0,0,1],[0,0,0,0,1,0,0],[0,0,0,0,1,0,0],[0,0,0,0,0,0,1]]*One(K);
	
	#Construct the semigroup S generated by x and y
	S:=Semigroup(x,y);
	
	
	# Construct the algebra A generated by x and y
	A := Algebra(K, [x, y]);
	
	
	# Check if every element in S satisfies X^2 = X
	
	v:=true;
	for i in S do
	if not IsZero(i^2-i) then v:=false; break; fi;
	od;
	
	if v=false then
	Print("\nThis algebra doesn't satisfy the given condition!\n");
	else
	Print("\nThis algebra satisfy the given condition!\n");
	fi;
	
	# Get a basis B for the algebra A
	B := Basis(A);
	
	# Compute the number of elements and dimension of A
	s := Size(A);
	t := Size(B);
	
	# Print size and dimension
	Print("\nSize of  A:", s, "\n");
	Print("\nDimension of A as a vectorial space:", t, "\n");
	
	#Creates the vector subspace over K generated by { x, y, xy, yx, xyx, 
 yxy}

	Sub:=VectorSpace(K,[x,y,x*y,y*x,x*y*x,y*x*y]);
	
	#Computes the dimension of the subspace Sub
	dim := Dimension(Sub);
	
	Print("\n Dimension of the subspace generated by {x, y,xy, yx, yxy}: ",
 dim, "\n");
\end{verbatim}

\section{Verification Code for Theorem \ref{t42}}

\begin{verbatim}
	# Define the finite field GF(5)
	K:=GF(5);
	
	# Left regular representation of the two generators of the algebra 
	# as two 4x4 matrices over GL(5)
	x:=[[0,1,0,0],[-1,2,0,0],[0,0,0,1],[0,0,-1,2]]*One(K);
	y:=[ [0,0,1,0],[-2,2,2,-1],[-1,0,2,0],[-2,1,2,0]]*One(K);
	
	#Construct the semigroup S generated by x and y
	S:=Semigroup(x,y);
	
	
	# Construct the algebra A generated by x and y
	A := Algebra(K, [x, y]);
	# Check if every element in S satisfies X^2 = 2X - 1
	
	v:=true;
	for i in S do
	if not IsZero(i^2-2*i+One(A)) then v:=false; break; fi;
	od;
	
	
	if v=false then
	Print("\nThis algebra doesn't satisfy the given condition!\n");
	else
	Print("\nThis algebra satisfy the given condition!\n");
	fi;
	
	# Get a basis B for the algebra A
	B := Basis(A);
	
	# Compute the number of elements and dimension of A
	s := Size(A);
	t := Size(B);
	
	# Print size and dimension
	Print("\nSize of  A:", s, "\n");
	Print("\nDimension of A as a vectorial space:", t, "\n");
	
	#Creates the vector subspace over K generated by { 1, x, y, xy }
	Sub:=VectorSpace(K,[One(A),x,y,x*y]);
	
	#Computes the dimension of the subspace Sub
	dim := Dimension(Sub);
	
	Print("\n Dimension of the subspace generated by {1, x, y, xy}: ", 
 dim, "\n");
\end{verbatim}

\section{Verification Code for Theorem \ref{t43}}

\begin{verbatim}
#Define the finite field GF(3)
K:=GF(3);

#Left regular representation of the three generators of the algebra
#as three matrices over GL(3)
x := [[0,1,0,0,0,0,0,],
      [-1,2,0,0,0,0,0,],
      [0,0,0,0,1,0,0,],
      [0,0,0,0,0,1,0,],
      [0,0,-1,0,2,0,0,],
      [0,0,0,-1,0,2,0,],
      [1,-1,-1,-1,1,1,1]] * One(K);

y := [[0,0,1,0,0,0,0,],
      [-2,2,2,0,-1,0,0],
      [-1,0,2,0,0,0,0,],
      [0,0,0,0,0,0,1],
      [-2,1,2,0,0,0,0],
      [-1,1,1,-1,-1,1,1],
      [0,0,0,-1,0,0,2]] * One(K);

z := [[0,0,0,1,0,0,0],
      [-2,2,0,2,0,-1,0],
      [-2,0,2,2,0,0,-1],
      [-1,0,0,2,0,0,0],
      [-3,1,1,3,1,-1,-1],
      [-2,1,0,2,0,0,0],
      [-2,0,1,2,0,0,0]] * One(K);
# Construct the semigroup S generated by x, y, and z
S := Semigroup(x, y, z);

# Construct the associative algebra A generated by x,y, and z over GF(3)
A := Algebra(K, [x, y, z]);

# Check if every element in S satisfies X^2 = 2*X - 1 (as in Theorem t43)
v := true;
for i in S do
	if not IsZero(i^2 - 2*i + One(A)) then
		v := false;
		break;
	fi;
od;

if v = false then
	Print("\nThis algebra does not satisfy the given quadratic condition!\n");
else
	Print("\nThis algebra satisfies the given quadratic condition!\n");
fi;

# Get a basis B for the algebra A
B := Basis(A);

# Compute the number of elements and dimension of A
s := Size(A);
t := Size(B);

# Print size and dimension
Print("\nSize of A: ", s, "\n");
Print("\nDimension of A as a vector space: ", t, "\n");

# Create the subspace generated by {1, x, y, z, xy, xz, yz}
Sub := VectorSpace(K, [One(A), x, y, z, x*y, x*z,y*z]);

# Compute the dimension of the subspace Sub
dim := Dimension(Sub);

Print("\nDimension of the subspace generated by
{1, x, y, z, xy, xz, yz}: ", dim, "\n");


\end{verbatim}

\scriptsize\--------------------------------------------------------------------------------------------------------------\\

\noindent\tiny{\bf{DISCLAIMER}}\\
This chapter is an extended version of the article published by the same author(s) in the following journal. Asian Research Journal of Mathematics, 21(6): 109-125, 2025 . DOI:10.9734/arjom/2025/v21i6947 \\ 
Available: https://journalarjom.com/index.php/ARJOM/article/view/947\\ [2mm]
\end{document}